\documentclass{amsart}[12pt]
\usepackage{setspace}
\usepackage[english]{babel}
\usepackage{framed}
\usepackage{soul}
\usepackage[margin=1in]{geometry}
\usepackage{arydshln}
\usepackage[colorinlistoftodos]{todonotes}
\usepackage{pb-diagram}
\usepackage[mathscr]{euscript}
\usepackage{multicol}
\usepackage{stmaryrd} 
\usepackage{empheq} 
\usepackage{tikz-cd}
\tikzset{
  symbol/.style={
    draw=none,
    every to/.append style={
      edge node={node [sloped, allow upside down, auto=false]{$#1$}}}
  }
}

\usepackage{amssymb,amsfonts,amsmath,amsopn,amstext,amscd,color,latexsym,mathrsfs,skak,stackrel,url,verbatim,amsthm}
\usepackage{mathabx}
\usepackage[utf8]{inputenc}
\usepackage[all, cmtip]{xy}
\usepackage{mathtools}
\usepackage{enumitem}
\usepackage[active]{srcltx}

\usepackage[pdftex, colorlinks=true, linktoc=all]{hyperref}

\usepackage{cleveref}

\theoremstyle{plain}
\newtheorem{thm}{\bf Theorem}[section]
\newtheorem{lemma}[thm]{\bf Lemma}
\newtheorem{cor}[thm]{\bf Corollary}
\newtheorem{prop}[thm]{\bf Proposition}

\theoremstyle{remark}

\newtheorem{example}[thm]{\bf Example}

\newtheorem{remark}[thm]{\bf Remark}

\theoremstyle{definition}
\newtheorem{definition}[thm]{\bf Definition}

\numberwithin{equation}{thm}

\newcommand{\C}{{\mathbb C}}

\newcommand{\bbP}{{\mathbb P}}
\newcommand{\bbQ}{{\mathbb Q}}

\newcommand{\Z}{{\mathbb Z}}

\newcommand{\cG}{{\mathcal G}}

\newcommand{\cL}{{\mathcal L}}

\newcommand{\rE}{{\rm E}}

\newcommand{\rH}{{\rm H}}

\newcommand{\Aut}{{\rm Aut}}

\newcommand{\End}{{\rm End}}

\newcommand{\Hom}{{\rm Hom}}

\DeclareMathOperator{\Sym}{Sym}

\newcommand{\isom}{\cong}

\begin{document}
\title{Algebraicity of Hodge classes on some Generalized Prym Varieties}
\author{Deepam Patel}
\address{Department of Mathematics, Purdue University,
150 N. University Street, West Lafayette, IN 47907, U.S.A.}
\email{patel471@purdue.edu}
\author{Yilong Zhang}
\address{Department of Mathematics, University of Georgia,
110 Carlton St, Athens, GA 30602, U.S.A.}
\email{Yilong.Zhang@uga.edu}

\date{May 22, 2026}

\thanks{D.P. would like to acknowledge support from the Simons Foundation}

\begin{abstract}
In this article, we revisit Schoen’s construction of algebraic cycles on certain Prym varieties \cite{Schoen88}. More precisely, we show that these cycles arise naturally from (unramified) geometric class field theory, and use this perspective to prove the algebraicity of certain Hodge classes on some {\it generalized Prym varieties}. 
\end{abstract}
\maketitle

\tableofcontents

\section{Introduction} 
We work over the field of complex numbers, and all varieties will be considered over this field. Given a curve $C$, let $J(C)$ denote the Jacobian variety of $C$ and $\Sym^k(C)$ denote its $k$-th symmetric power. We shall assume that a point $c_0 \in C$ is fixed, and consider the corresponding Abel--Jacobi maps (recalled below).

As an easy consequence of basic covering space theory, one obtains that every abelian cover of smooth projective curves $C\to C'$, with covering group $G$, arises from base change of finite covers of Jacobian varieties. In particular, there is a Cartesian diagram 
 \begin{equation}\label{diagram_intro-GCFT}
\begin{tikzcd}
C \arrow[d] \arrow[r,hookrightarrow]& W\arrow[r]\arrow[d]&A\arrow[d]\\
C' \arrow[r] \arrow[r,hookrightarrow,"\Delta_d"]&\Sym^d(C')\arrow[r,"AJ_d"]& J(C'),
\end{tikzcd}
\end{equation}
where the first map on the bottom row is the `diagonal map', $AJ_d$ is the degree $d$ Abel--Jacobi map (note that this depends on the fixed point $c_0' \in C'$), $A\to J(C')$ is an isogeny of abelian varieties, and the diagram is $G$-equivariant.  We will show how this diagram produces nontrivial algebraic cycles for suitable values of \(d\). The existence of the diagram above can also be viewed as an instance of unramified geometric class field theory for curves, and is valid over arbitrary algebraically closed fields. From this perspective, one has analogous assertions in the {\it ramified} setting, where one considers instead Jacobians with modulus. However, we do not discuss the ramified setting in this article.

It is classical (for example, see \cite{ACGH85}) that for $d$ sufficiently large ($d\ge 2g-1$, where $g=g(C')$), the Abel--Jacobi map
$$AJ_d: \Sym^d(C') \rightarrow J(C')$$
is a projective bundle. On the other hand, if $d=2g-2$ then the Abel--Jacobi map is a $\mathbb P^{g-2}$-bundle over the complement of a point and the fiber over this point is $\mathbb P^{g-1}\cong |K_{C'}|$ coming from the canonical divisor class on $C'$. The cycle $\mathbb P^{g-1}$ defines a class in the middle-dimensional cohomology, and then the pullback of $\mathbb P^{g-1}$ to $\Sym^{2g-2}(C)$ produces interesting algebraic cycles. It turns out that these cycles factor through $W$, an abelian cover of $\Sym^{2g-2}(C')$. 

\subsection{Cohomology on $G$-cover}  The following result describes the cohomology of the abelian cover $W$ of the symmetric product of a curve, and in particular shows the algebraicity of certain Hodge classes. 
\begin{thm}\label{thm:mainthm1}
Let $C\to C'$ be an \'etale $G$-cover of smooth projective curves, where the Galois group $G$ is a finite abelian group. Let $h=2g(C')-2$ and suppose $g(C')\ge 2$. 

\begin{enumerate}

\item There is a natural isomorphism of Hodge structures
\begin{equation}\label{eqn_Wcoho}
    \rH^{n}(W,\bbQ)\cong \begin{cases}
        \rH^{h}(\Sym^h(C'),\bbQ)\oplus U(-\frac{h}{2}),\ \textup{when}\ n=h;\\
        \rH^{n}(\Sym^h(C'),\bbQ),\ \textup{when}\ n\neq h.\\
    \end{cases}
\end{equation}

\item The summand $U(-\frac{h}{2})$ is generated by algebraic classes, and has dimension $|G|-1$.

\item The above isomorphism is also $G$-equivariant, where the $G$-action on $\rH^*(\Sym^h(C'))$ is trivial, and the action on $U$ is isomorphic to the complement of the trivial factor in the regular $G$-representation. In particular, its complexification decomposes as a sum of $\chi$-isotypic components
$$U\otimes \C\cong \bigoplus_{\chi\neq 0}\rH^{h}(W,\C)_{\chi},$$
where each summand is one-dimensional and $\chi$ ranges over all non-trivial characters.

\end{enumerate}
\end{thm}

On the other hand, the finite cover $C\to C'$ induces the Norm map on Jacobian varieties $J(C)\to J(C')$. The connected component $B$ of the identity in its kernel is an abelian subvariety of $J(C)$ called the \textit{Prym variety}. There is an isogeny of abelian varieties $$J(C) \sim J(C') \times B.$$ At the cohomological level, this corresponds to the decomposition $$\rH^1(J(C),\bbQ)\cong  \rH^1(J(C'),\bbQ) \oplus \rH^1(J(C),\bbQ)_{nt},$$ 
where the summand $\rH^1(J(C),\bbQ)_{nt}$ is naturally identified with $H^1(B,\bbQ)$ and has the property that its complexification is a sum of $\chi$-isotypic components ranging over all non-trivial characters $\chi$ of $G$.

\subsection{Algebraicity of Weil Hodge classes}
Since $G$ is abelian, the group algebra $\bbQ[G]$ is semi-simple and is a product of fields $F_1\times \cdots\times F_r$, where each field $F_i$ corresponds to a rational irreducible representation of $G$. In particular, there is a decomposition of rings
$\bbQ[G]\cong \bbQ\times \bbQ[G]_{nt}$,
where the ring $\bbQ[G]_{nt}$ is the product of fields corresponding to non-trivial representations. 


It follows that the first cohomology group of the Prym variety $\rH^1(B,\bbQ)$ is naturally a $\bbQ[G]_{nt}$-module. In fact, since each non-trivial character will appear exactly $h$ times, $\rH^1(B,\bbQ)$ is a free $\bbQ[G]_{nt}$-module with rank $h$, and the resulting top exterior power 
\begin{equation}\label{eqn_intro_Weil-classes}
    U_{Weil}:=\bigwedge^h_{\bbQ[G]_{nt}}\rH^1(B,\bbQ) \subset \rH^h(B,\bbQ)
\end{equation}
consists of Hodge classes (see Section \ref{sec_prelim} for the proof of these statements). This is analogous to the construction of Weil Hodge classes in the setting of Weil abelian varieties. Note that this is a subspace of dimension $|G|-1$, and is in general not generated by the divisor classes. Our second result is

\begin{thm}\label{thm:mainthm2}
The Hodge classes $U_{Weil}$ \eqref{eqn_intro_Weil-classes} are represented by algebraic cycles. 
\end{thm}
As Hodge substructures, $U_{Weil}$ is identified with $U$ from \eqref{eqn_Wcoho} by pulling back to $\Sym^h(C)$, applying Abel--Jacobi map, and powers of the Lefschetz operator on $H^*(B,\bbQ)$. The algebraicity follows from the same chain of maps, together with Theorem \ref{thm:mainthm1} and the standard conjecture on abelian varieties \cite{Kleiman}.

In fact, one can associate abelian subvarieties $J(C)_V$ to any rational representation $V$ of $G$, and consider an analogous space of Hodge classes. Theorem \ref{thm:mainthm2} also holds in this setting with essentially the same proof.

\subsection{Relation to Schoen's work}
In the case of cyclic étale covers, after restricting to the rational representation associated with the {\it primitive characters}, Theorem \ref{thm:mainthm2} is due to Chad Schoen \cite{Schoen88}.
In that setting, Schoen showed that there is a canonical abelian subvariety of $B$ corresponding to primitive characters of $\Z/m\Z$. When $m=3$ (resp.  $4$), these give rise to Weil abelian fourfolds with CM field $\bbQ(\mu_3)$ (resp. $\bbQ (i)$) with trivial discriminant. Schoen proved the Hodge conjecture for such fourfolds by specializing to the ones arising from the Prym construction as above. Another proof for the $\bbQ(i)$ case was given by van Geemen \cite{vGQ(i)}. Note that based on the work of Faber and Koike \cite{Faber, Koike}, Schoen's result \cite{Schoen88, Schoen98} shows the Hodge conjecture for Weil abelian fourfolds with CM field $\bbQ(\mu_3)$ and $\bbQ (i)$ and arbitrary discriminant. By recent work of Markman \cite{Markman23, Markman25}, the Hodge conjecture is now known for all Weil abelian fourfolds with CM in other quadratic fields.

The main goal of this paper is to generalize Schoen’s method for constructing algebraic cycles to the case of an abelian cover, that is, an étale cover with abelian Galois group. The results above follow in part Schoen's original method of proof. The main contribution here is to put forth the viewpoint that the required algebraic cycles come from (unramified) geometric class field theory. Specifically, the key statement for the algebraicity of Hodge classes is Theorem \ref{thm:mainthm1}, which in turn follows easily from the geometry of the Abel--Jacobi map. 

It should be noted that even in the cyclic case, Schoen only considers the rational representation associated to primitive characters, whereas here we consider all non-trivial characters. On the other hand, it is possible that Schoen's methods can be adapted to this more general setting. The main goal here is to provide an alternate viewpoint that generalizes and simplifies Schoen's construction of algebraic cycles. Moreover, the geometric class field theory perpective allows one to easily generalize many of the results in this article to the ramified setting, where one works systematically with Jacobians with modulus. This approach also partially generalizes to surfaces (and $\rH^2(S,\bbQ)$ where $S$ is a surface). 

We do not consider the ramified or higher-dimensional cases here, and plan to discuss these in a future paper. Finally, we note that the proofs here adapt easily to the \'etale cohomology setting and Tate classes (though we stick to the complex setting below). \\

{\bf Contents.} In Section \ref{sec_prelim}, we recall various preliminaries including Abel--Jacobi maps, representation theory, and how it leads to the construction of Prym varieties and the Weil Hodge classes. In Section \ref{sec_proof-thm1}, we study the cohomology of the abelian cover $W$ of the symmetric product of a curve and prove Theorem \ref{thm:mainthm1}. In Section \ref{sec_proof-thm2}, we relate the algebraic cycles on the abelian cover $W$ to the cycles on Prym varieties and prove Theorem \ref{thm:mainthm2}. In Section \ref{sec_cyclic}, we recall Schoen's examples in the setting of cyclic covers.\\

\noindent\textbf{Acknowledgement}. The first author thanks Madhav Nori, who first suggested that one ought to try to understand Schoen's paper from the perspective of geometric class field theory. We thank Chad Schoen and Bert van Geemen for helpful correspondence.

\section{Preliminaries}\label{sec_prelim}

\subsection{Geometric Class Field Theory}
In this section, we recall a basic result from geometric class field theory that will be used later. We refer to \cite{Serre-CFT} for details (where far more general results are proved). In fact, in this article, we shall only be concerned with the unramified setting, where the results we need are a simple consequence of covering space theory. These results are well known and are recalled here for the reader’s convenience.

Let $C'$ be a smooth projective curve of genus $g$ over a perfect field $k$. Let $\Sym^d(C')$ denote the $d$-th symmetric power of $C'$, and $C'^{d} \xrightarrow{\pi_d} \Sym^d(C')$ denote the natural projection. Finally, let ${\rm Pic}^0(C')$ denote the usual Picard scheme of degree $0$ line bundles. We assume that a rational point $P \in C'(k)$ is fixed. With this notation, we have the standard morphisms:
$$C' \xrightarrow{\Delta_d} \Sym^d(C') \xrightarrow{AJ_d} {\rm Pic^0}(C'),$$
where the first arrow is given by composing $\pi_d$ with the diagonal map, and the second is the degree $d$ Abel--Jacobi map, which depends on the fixed point $P$. With this notation, one has the following:

\begin{thm}\label{thm:gcft}
\begin{enumerate}
\item Both $\Delta_d$ and $AJ_d$ induce isomorphisms on the abelianizations of geometric \'etale fundamental groups (resp. topological fundamental groups when $k=\C$). Note that both ${\rm Sym}^d(C')$ ($d \geq 2$) and ${\rm Pic^0}(C')$ have abelian \'etale (resp. topological) fundamental groups.
\item In particular, every abelian \'etale covering $C \rightarrow C'$ fits into a commutative Cartesian diagram:
\begin{equation}\label{diagram_GCFT}
\xymatrix{
C  \ar[r] \ar[d] & W  \ar[r] \ar[d] & A \ar[d] \\
C' \ar[r] & \Sym^d(C') \ar[r] & {\rm Pic^0}(C'),}  
\end{equation}
where the vertical maps are all finite \'etale abelian covers, and both $W$ (resp. $A$) are unique up to an isomorphism.
\end{enumerate}
\end{thm}    
\begin{proof}
For the first part, note that this is by construction in the case of Abel--Jacobi map. We refer to \cite{sympi1} for the analogous result for $AJ_d$ ($d \geq 2$), and the fact that the fundamental group of $\Sym^d(C')$ is abelian. The second part follows from the first as an application of standard covering space theory. Specifically, we have a bijection between subgroups of the abelianization of the fundamental group and abelian covers.
\end{proof}

We also record the following basic fact about the Abel--Jacobi morphism. 
\begin{thm}(cf. \cite{Schwarzenberger63})\label{thm:abeljacobi} Let $C'$ be a smooth projective curve over an algebraically closed field $k$.
\begin{enumerate}
\item For $d \geq 2g-1$, the Abel--Jacobi map $AJ_d$ is a projective bundle.
\item If $d = 2g-2$, then $AJ_d$ is a projective bundle with fibers $\bbP^{g-2}$ over an open subset $\mathcal{U}$, where the complement $Z=J(C')\setminus \mathcal{U}$ is a point. Over this point, the fiber is $\mathbb P^{g-1}$, corresponding to the canonical divisor class.  
\end{enumerate}
\end{thm}

Note that set-theoretically, this follows from Abel's theorem and the Riemann--Roch theorem: the fiber over $AJ_d(D-dP)$ is the projectivization of the linear system $|D|$ with dimension $\dim|D|-1=d-g-h^1(D)$. Then by Serre duality $h^1(D)=h^0(K_{C'}-D)$, which is zero when $\textup{deg}(D)\ge 2g-1$; when $\textup{deg}(D)=2g-2$, it vanishes if and only if $D$ is the canonical divisor. 

For the rest of the paper, we work with $k=\C$.

\subsection{A Prym type construction}\label{subsec_Prym-by-rep}
Given a finite group $G$, a rational representation $V$ of $G$, and a group homomorphism $G \rightarrow {\rm End}_0(A) := \End(A) \otimes \bbQ$ , one can associate an abelian subvariety $A_V \subset A$. We briefly recall this construction following (\cite{LR}). 

With $G$ as above, consider the corresponding group rings $\bbQ[G]$ and $\C[G]$. Since $\bbQ[G]$ is semi-simple, we may write $$\bbQ[G] = \oplus_{i=1}^s \bbQ[G]e_{V_i}$$ as a direct sum of simple modules where $e_{V_i}$ are central idempotents and $V_1,\ldots,V_s$ are the irreducible rational representations of $G$. As a representation, $\bbQ[G]e_{V_i} \simeq V_i^{m_i}$, where $m_i = \frac{\dim(V'_i)}{s_i}$, $V'_i$ is a complex irreducible representation associated with $V_i$, and $s_i$ is the Schur index (\cite{LR}). One has a similar description of $\C[G]$ in terms of irreducible complex representations. Note that if $G$ is abelian, then $\C[G]$ is a sum of characters, and $s_i = 1$.

If $\alpha \in {\rm End_0}(A)$, then $m\alpha \in {\rm End}(A)$ for some positive integer $m$. We define $A_{\alpha} := \textup{Im}(m\alpha)$ for a minimal such $m$. In particular, we have abelian subvarieties $A_{V_i} \subset A$ associated with each $V_i$ (given by $\textup{Im}(me_{V_i})$ for appropriate $m$). 

\begin{definition}\label{def_A_V}
   If $V$ is a rational representation of $G$, then we may write $V = \sum_i n_iV_i$ as a sum of irreducible representations, and we set $$A_V: = \prod_{n_i > 0} A_{V_i}.$$ 
\end{definition}

\begin{thm}(\cite{LR}, 2.9.1)
\begin{enumerate}
\item Each $A_{V_i}$ is $G$-stable, the $A_{V_i}$ are pairwise orthogonal, and the addition map induces a $G$-equivariant isogeny $\prod A_{V_i} \rightarrow A$.
\item The $G$-action on $A_{V_i}$ is through a multiple of the rational representation $V_i$.
    \end{enumerate}
\end{thm}

In particular, for any rational representation $V$ of $G$ we obtain a $G$-equivariant isogeny $A_{V^c} \times A_V \rightarrow A$ where $V^c$ is the sum of the irreducible rational representations not appearing in $V$. The following two examples will be useful later.

\begin{example}\label{example_A-from-rep}
\begin{enumerate}
\item Let $V_{nt} = \oplus V_i$ where the sum is over all the non-trivial rational representations. Let $A_{nt} \subset A$ denote the corresponding variety. On cohomology, we have $$\rH^1(A,\bbQ) = \rH^1(A,\bbQ)^{G} \oplus \rH^1(A_{nt},\bbQ),$$
where $\rH^1(A_{nt},\C)$ decomposes as a direct sum of $\chi$-isotypic components of $\rH^1(A,\bbQ)$, where $\chi$ ranges over non-trivial characters of $G$.

\item Let $G$ be cyclic of prime order. Let $V$ be the rational representation whose complexification is the sum of the $\chi$-isotypic components where $\chi$ varies over primitive characters. In this setting, the resulting $A_V$ is what is referred to as the Prym variety in the literature. (Usually, in that setting, $A$ is taken to be the Jacobian of a $G$-covering of some curve.) 
\end{enumerate}
\end{example}
We will investigate these two cases in Sections \ref{subsec_WeilHC-Prym} and \ref{sec_cyclic} for coverings of curves.
\subsection{Weil type classes}\label{subsec:weiltypeclasses}
Suppose a finite group $G$ acts on $A$, and let $\rho: \bbQ[G] \rightarrow \End(A)_{\bbQ}$ denote the corresponding rational representation as above. Then the simple algebras (indexed by the irreducible rational representations $V_i$) appearing in the decomposition of $\bbQ[G]$ are all matrix algebras over division algebras. We denote by $F_i$ the centers of these division algebras (corresponding to the representations $V_i$). In the current setting, this is a number field. Let $R = F_1 \times \cdots \times F_k$ and $R \rightarrow \End(A)_\bbQ$ be the induced morphism. It follows that $V:= \rH^1(A,\bbQ)$ is a finitely generated $R$-module. We assume that it is a free $R$-module of rank $r$. Let $V_R := \Lambda^r_R V$ denote the resulting rank $1$-module; it is a $\dim_{\bbQ}(R)$-dimensional $\bbQ$-vector space. Note that one has a natural inclusion 
$$ V_R \xhookrightarrow{} \rH^r(A,\bbQ) = \Lambda_{\bbQ}^r V.$$
Let $\Sigma_R := \Hom_{\bbQ}(R,\mathbb{C}) = \Pi_i\Hom_{\bbQ}(F_i,\mathbb{C})$ (here we mean $\bbQ$-algebra homomorphisms), and for $\sigma \in \Sigma_R$ let $n_{\sigma}$ denote the multiplicity of $\sigma$ in the $R$-action on $V_{\mathbb{C}}^{(1,0)}$. We have the following standard criterion for when $V_R$ consists of Hodge classes. 

\begin{lemma}\label{lemma_chi-chi-bar}
If $n_{\sigma} = n_{\bar{\sigma}}$ for all $\sigma \in \Sigma_R$, then $V_R$ consists entirely of Hodge classes.
\end{lemma}
\begin{proof}
If $R$ is a field, this is well-known (\cite{MZ}). That argument also applies to a more general setting where $R$ is a product of fields as we consider here. 
\end{proof}

In the following, we shall be concerned with the setting where $G$ is abelian. In that case, each $F_i$ is a subfield of the cyclotomic field. If $R =F$ is a field, then since $V_F$ is a rank one $F$-module, it follows that $V_F$ consists entirely of either exceptional Hodge classes, or entirely of those coming from divisors.

\subsection{Local systems on finite \'etale covers}\label{subsec:localsyscovers}
Let $\pi: X \rightarrow X'$ be a finite (Galois) \'etale cover of a complex algebraic variety with Galois group $G$. Then we have the standard short exact sequence of fundamental groups (we ignore base points for ease of notation):
$$1 \rightarrow \pi_1(X) \rightarrow \pi_1(X') \rightarrow G \rightarrow 1.$$

A local system $\cG$ (of $k$-vector spaces, with $k$ a field) on $X$ is equivalent to giving a finite-dimensional $\pi_1(X)$-representation in a $k$-vector space. Let $\rho_{\cG}$ denote this representation. Then the pushforward $\pi_*\cG$ corresponds to the induced representation 
${\rm Ind}_{\pi_1(X)}^{\pi_1(X')} \rho_{\cG}$ (as a consequence of standard covering space theory).

A finite-dimensional representation of $G$ gives rise to a local system on $X'$; in particular, every (complex) character $\chi$ of $G$ gives a (rank one) local system $\cL_{\chi}$ on $X'$. We may also express these constructions in the language of group rings. More precisely, a finite-dimensional $k$-representation $V$ of $\pi_1(X)$ is equivalent to a $k[\pi_1(X)]$-module. We sometimes use the same notation $V$ to denote the corresponding module. For example, the constant local system $k$ is given by the quotient $k[\pi_1(X)]/I$ where $I$ is the augmentation ideal. The induced representation is then given by the tensor product $k[\pi_1(X')] \otimes_{k[\pi_1(X)]} V.$ We have the following standard computation. 

\begin{lemma}\label{lemma_H^i(W)-decomp}
With the same notation as above, assume moreover $G$ is abelian. 
\begin{enumerate}
\item One has a natural isomorphism of local systems on $X'$
\begin{equation}\label{eqn_pi_*C-decomp}
  \pi_*(\mathbb{C}) \isom \oplus_{\chi} \cL_{\chi},  
\end{equation}
where the sum is over all characters of $G$.
\item Taking cohomology, there is a decomposition $$\rH^i(X,\mathbb{C}) \isom \bigoplus_{\chi} \rH^i(X,\mathbb{C})_{\chi},$$ 
where $\rH^i(X,\mathbb{C})_{\chi}\cong \rH^i(X',\cL_{\chi})$ is the $\chi$-isotypic component of the $G$-representation $\rH^{i}(X,\mathbb{C})$.
\end{enumerate}
\end{lemma}
\begin{proof}
By the remarks above, the push-forward corresponds to the module $\mathbb{C}[\pi_1(X')] \otimes_{\mathbb{C}[\pi_1(X)]} \mathbb{C}[\pi_1(X)]/I.$ On the other hand, the latter is easily seen to be isomorphic to $\mathbb{C}[G]$ (as a $\mathbb{C}[\pi_1(X')]$-module). Finally, note that $\mathbb{C}[G]$ (as a $\mathbb{C}[G]$-module) is a direct sum of the characters of $G$. This follows from the standard fact that the regular representation of a finite abelian group is the direct sum of all the one-dimensional representations, each appearing with multiplicity one. The second part is a direct consequence of the first part, since the $\chi$-isotypic component is precisely $\rH^{i}(X',\cL_{\chi}).$
\end{proof}

\begin{remark}
We make a clarifying remark with regard to the {\it naturality} of the isomorphism above. Given the following Cartesian square.
\begin{figure}[ht]
    \centering
\begin{equation}
\begin{tikzcd}
Y \arrow[r]\arrow[d,"\pi_Y"]&  X\arrow[d,"\pi_X"]\\
Y' \arrow[r,"f"]&X'.
\end{tikzcd}
\end{equation}
\end{figure}
The decompositions \eqref{eqn_pi_*C-decomp} for $\pi_X$ and $\pi_Y$ are compatible with the pullback along $f$.
\end{remark}

\subsection{The group action on a curve} We specialize to the setting of an \'etale abelian cover of curves $\pi: C \rightarrow C':= C/G$. As an application of the results of the previous section, we have 
$$\rH^1(C,\C) = \bigoplus_{\chi} \rH^1(C',\cL_{\chi}),$$
where $\rH^1(C',\cL_{\chi})$ is the $\chi$-isotypic component (cf. Lemma \ref{lemma_H^i(W)-decomp}). We are interested in determining the multiplicity of this component, i.e. the multiplicity of $\chi$ in the $G$-representation $\rH^1(C,\mathbb{C})$. It is well known (see \cite[Lemma 1.5]{Schoen88}) that this is precisely $h:= 2g-2$ where $g$ is the genus of $C'$. Here we give a simple direct proof (without recourse to the Lefschetz fixed point formula as in \cite[Lemma 1.5]{Schoen88}).

\begin{lemma}\label{lem:chimultiplicity}
With notation as above, $\dim\rH^1(C',\cL_{\chi}) = 2g - 2$, if $\chi$ is not trivial. 
\end{lemma}
\begin{proof}
First, it follows from standard topology that the Euler characteristic of a complex local system $\cL_{\chi}$ with rank one is the same as the topological Euler characteristic $\chi(C',L_{\chi})=\chi(C')=2-2g$. Then since $\cL_{\chi}$ has non-trivial monodromy, $\dim\rH^0(C',\cL_{\chi}) = \dim\rH^2(C',\cL_{\chi}) = 0$, hence the result follows.
\end{proof}

Now, we may split the local system into $\pi_*\bbQ = \bbQ \oplus \cL_{nt}$. Recall that the local system $\cL_{nt}$ corresponds to the sum of the non-trivial irreducible rational representations of $G$, and each of these has a field $F_i$ associated with it. Let $\bbQ[G]_{nt}$ denote the product of these fields. Then $\cL_{nt}$ has an action of $\bbQ[G]_{nt}$, and an induced action on $\rH^{1}(C',\cL_{nt}) = \rH^1(C,\bbQ)_{nt}$. The cohomology group $\rH^1(C,\bbQ)_{nt}$ is naturally a $\bbQ[G]_{nt}$-module. As a consequence of Lemma \ref{lem:chimultiplicity}, we have
\begin{cor}\label{cor:freemodule}
With notation as above, $\rH^1(C,\bbQ)_{nt}$ is a free $\bbQ[G]_{nt}$-module of rank $h$. 
\end{cor}

\begin{remark}\label{rem:freemodulegeneral}
In fact, the analog of the previous lemma is true for any rational $G$-representation. More precisely, let $\rH^1(C,\bbQ)_V$ denote the $V$-isotypic component. Then $V$ is a sum of a collection of irreducible rational representations, and therefore one has a ring $R_V$ associated to $V$ (i.e. the product of fields corresponding to the irreducible summands appearing in $V$). Then $\rH^1(C,\bbQ)_V$ is a free $R_V$-module of rank $h$.
\end{remark}

\subsection{Weil Hodge classes on Prym variety}\label{subsec_WeilHC-Prym}
In this section, we will study Example \ref{example_A-from-rep} (1) in the case where the abelian variety is $J(C)$, and the representation is $V_{nt}$. We work in the same setting as before, where $C\to C'$ is an \'etale $G$-cover of curves. The {\it Prym variety} is the connected component containing the identity of the kernel of the Norm map 
\begin{equation}\label{eqn_Prym}
    B:=\ker(\textup{Nm}:J(C)\to J(C'))_0.
\end{equation}
With notation in Definition \ref{def_A_V}, this is $J(C)_{nt}$. We have an identification $\rH^1(B,\bbQ)\cong \rH^1(J(C),\bbQ)_{nt}$. Recall that this is an $h$-dimensional free module over the ring $\bbQ[G]_{nt}$ by Corollary \ref{cor:freemodule}. Hence, the top exterior power 
\begin{equation}\label{eqn_UWeil}
    U_{Weil} := \bigwedge^h_{\bbQ[G]_{nt}} \rH^1(B,\bbQ)\subseteq \rH^h(B,\bbQ)
\end{equation} 
defines a subspace which is
\begin{itemize}
    \item $|G|-1$ dimensional over $\bbQ$, and
    \item consists of Hodge classes (cf. Lemma \ref{lemma_chi-chi-bar}).
\end{itemize}

We can describe the complex vector space $U_{Weil}\otimes \C$ as follows. For each nonzero character $\chi$, we choose a basis $v_1,\ldots,v_h\in \rH^1(C,\C)_{\chi}$ in the $\chi$-isotypic component, then $v_1\wedge\cdots\wedge v_h\in U_{Weil}\otimes \C$. When $\chi$ runs through all nonzero characters, such vectors form a basis of $U_{Weil}\otimes \C$ (note that a wedge product $v_1^{\chi_1}\wedge\cdots\wedge v_h^{\chi_h}$ is zero, where $v_i^{\chi_i}\in \rH^1(C,\C)_{\chi_i}$ nonzero, unless $\chi_1=\cdots=\chi_h$).

We will refer to the classes given by \eqref{eqn_UWeil} as Weil Hodge classes and denote this subspace by $U_{Weil}$. We may consider this a subspace of the cohomology of ${\rm Sym}^h(C)$ by pulling back via the composition of Abel--Jacobi map $\Phi:\Sym^h(C)\to J(C)$ and the projection $J(C)\to B$
$$\rH^h(B)\to \rH^h(J(C),\bbQ)\xrightarrow{\Phi^*} \rH^h(\Sym^h(C),\bbQ).$$

Using the fact that there is an isomorphism $\Lambda^hH^1(C,\bbQ)\cong H^h(J(C),\bbQ)$, and a theorem of MacDonald (\cite{Macdonald}) that there is an isomorphism of Hodge structures $$\rH^h({\rm Sym}^hC,\bbQ)\cong \bigoplus_{0\le i\le h/2}\bigwedge^{h-2i} \rH^1(C,\bbQ),$$ we may view 
$U_{Weil}$ as a Hodge substructure of the factor $\Lambda^h \rH^1(C,\bbQ)$ of $H^h(\Sym^h(C),\bbQ)$. 

\section{Cohomology on Abelian Cover of Symmetric Product}\label{sec_proof-thm1}
\subsection{Set-up}
Recall that we work with an unramified Galois cover $C\to C'$ of smooth projective curves with Galois group $G$ being abelian. We are interested in computing the cohomology of an \'etale $G$-cover of $\Sym^h(C')$, where $h=2g(C')-2$ is the degree of the canonical divisor, and also equals the multiplicity of each nontrivial character of \(G\) in \(\rH^1(C,\C)\) (cf. Lemma \ref{lem:chimultiplicity}).

By Theorem \ref{thm:gcft}, any \'etale $G$-cover of $\Sym^h(C')$ comes from a Cartesian square:
\begin{equation}\label{diagram_W-A}
\xymatrix{
W \ar[r] \ar[d]^{\tilde{\pi}} & A \ar[d]^{\pi}\\
{\rm Sym}^h(C') \ar[r]^{AJ_h} & J(C').}  
\end{equation}

Since $A$ is an abelian variety isogenous to $J(C')$, its cohomology (over $\bbQ$) will be isomorphic to that of $J(C')$. The rest of the section is devoted to computing the cohomology of $W$ and proving the first main theorem.
\subsection{Proof of Theorem \ref{thm:mainthm1}} 
We split the proof into three propositions. We assume that the base curve has genus $g(C')\ge 2$, since there is no non-trivial unramified cover of a genus zero curve; and an unramified cover of a genus one curve still has genus one, hence the Prym variety is trivial.
\begin{prop} (cf. Theorem \ref{thm:mainthm1}, (1) )\label{prop_Theorem1(1)}
There is a natural isomorphism of Hodge structures
\begin{equation}\label{eqn_prop3.1}
    \rH^{n}(W,\bbQ)\cong \begin{cases}
        \rH^{h}(\Sym^h(C'),\bbQ)\oplus U(-\frac{h}{2}),\ \textup{when}\ n=h;\\
        \rH^{n}(\Sym^h(C'),\bbQ),\ \textup{when}\ n\neq h,\\
    \end{cases}   
\end{equation}
where $U$ consists of Hodge classes.
\end{prop}

\begin{proof} We break up the proof in a few steps.
\begin{enumerate}
\item The direct image local system
$\pi_*\bbQ$ is a polarized variation of Hodge structure (of weight zero). We have a decomposition of this into a direct sum of local systems on $J(C')$ indexed by the irreducible rational representations of $G$. We may write this as $\bbQ \oplus \cL_{nt}$. (Note that this decomposition underlies variations of Hodge structures.) Here $\cL_{nt}$ is the rational local system which corresponds to the complement of the trivial representation. We have a similar decomposition $\tilde{\pi}_*\bbQ\cong \bbQ\oplus \tilde{\mathcal{L}}_{nt}$ on $\Sym^h(C')$. Moreover, by naturality, the pullback of $\cL_{nt}$ along $AJ_h$ is the corresponding local system on ${\rm Sym}^h(C')$, i.e., $AJ_h^*\mathcal{L}_{nt}=\tilde{\mathcal{L}}_{nt}$. Hence, we have an isomorphism
\begin{equation}\label{eqn_RGammaW-tilde}
    R\Gamma(W,\bbQ) \isom R\Gamma({\rm Sym}^h(C'),\bbQ) \oplus R\Gamma({\rm Sym}^h(C'),AJ_h^*\cL_{nt})
\end{equation}
in the (derived) category of mixed Hodge modules. 
\item We are interested in computing the second factor of \eqref{eqn_RGammaW-tilde}. By the projection formula, we have 
$$R\Gamma({\rm Sym}^h(C'),AJ_h^*\cL_{nt}) \isom R\Gamma(J(C'), (AJ_h)_*\bbQ \otimes \cL_{nt}).$$
\item Let $\mathcal{U} \xhookrightarrow{j} J(C')$ be the complement of the point $Z \xhookrightarrow{i} J(C')$ (as in Theorem \ref{thm:abeljacobi}). In particular, over $Z$, the morphism $AJ_h$ has fiber $\bbP^{g-1} = \bbP^{\frac{h}{2}}$ and over $\mathcal{U}$ it is a projective bundle of rank $g-2$. Let $\cG := (AJ_h)_*\bbQ$, then we have the distinguished triangle:
$$Rj_!\cG|_\mathcal{U} \rightarrow \cG \rightarrow i_*i^*\cG\xrightarrow{[1]} $$
By proper base change and the projective bundle formula, $\cG|_\mathcal{U} = \oplus_{j=0}^{g-2} \bbQ|_{\mathcal{U}}[-2j](-j).$
Similarly, we have $\cG|_Z = \oplus_{j=0}^{g-1} \bbQ|_{Z}[-2j](-j).$ 
Since both $j_!$ and $i^*$ are exact functors, taking the long exact sequence in cohomology for the above triangle gives short exact sequences:
$$0 \rightarrow j_!\bbQ|_{\mathcal{U}}(-k) \rightarrow R^{2k}(AJ_{h})_*\bbQ \rightarrow i_*\bbQ|_{Z}(-k) \rightarrow 0,$$
for $0 \leq k \leq g-2$. The higher direct images vanish in odd degrees since the stalks are cohomology of the fibers which are projective spaces. In the top degree, one has $$R^{2(g-1)}(AJ_{h})_*\bbQ \isom i_*\bbQ|_Z(-(g-1)).$$
\item Combining everything, we have 
\begin{equation} \label{eqn_AJ-pushforward}
R^{2k}(AJ_h)_*\bbQ  \isom 
\begin{cases}
   \bbQ(-k),\  \textup{for}\  0 \leq k \leq g-2, \\
    i_*\bbQ|_{Z}(-(g-1)),\ \textup{for}\ k=g-1,
\end{cases}  
\end{equation}
 and zero in other degrees.
\item  Note that 
\begin{equation}\label{eqn_J(C')torsor}
    \rH^k(J(C'),\cL_{nt}) = 0,\  \textup{for all}\ k.
\end{equation}
This is because after tensoring with $\mathbb{C}$, $\cL_{nt}\otimes \C$ splits as a sum of rank one local systems corresponding to the non-trivial characters of $G$. Each of these has finite monodromy given by roots of unity. On a torus, such local systems have vanishing cohomology. (Note this is consistent with the fact $\pi$ is an isogeny, and hence an isomorphism on cohomology.)

\item We have the Leray spectral sequence converging to $\rH^{k}({\rm Sym}^h(C'),\tilde{\mathcal{L}}_{nt})$, with $$\rE_{2}^{p,q} = \rH^p(J(C'),R^q(AJ_h)_*\bbQ \otimes \cL_{nt}).$$ By \eqref{eqn_AJ-pushforward} and \eqref{eqn_J(C')torsor}, all $\rE_2^{p,q}$ terms vanish except for $p=0$ and $q=2(g-1)$. It is precisely 
$\rH^0(J(C'), (i_*\bbQ|_Z)(-(g-1)) \otimes \cL_{nt}) =: U(-(g-1))$. It follows that the second term of \eqref{eqn_RGammaW-tilde} is
\begin{equation}\label{eqn_U-by-hypercohomology}
    \rH^h({\rm Sym}^h(C'),\tilde{\mathcal{L}}_{nt})\cong U(-(g-1)),
\end{equation}
and the cohomology vanishes at other degrees. Note that the isomorphism preserves the $G$-action. This completes the proof of the first part of Theorem \ref{thm:mainthm1}.
\end{enumerate}
\end{proof}

\begin{prop}(cf. Theorem \ref{thm:mainthm1}, (2) )\label{prop_Theorem1(2)}
    The Hodge substructure $U\subseteq \rH^{h}(W,\C)$ is spanned by algebraic classes, and has dimension $|G|-1$.
\end{prop}
\begin{proof}
From Proposition \ref{prop_Theorem1(1)}, the summand $U$ (up to Tate twist) is identified with $\rH^0\big(J(C'),\, i_* \mathbb{Q}_{|Z} \otimes \cL_{nt}\big)$, where $Z \subset J(C')$ is the point over which the Abel--Jacobi map has fiber
$Q = \mathbb{P}^{g-1}$, and $\cL_{nt}$ corresponds to the non-trivial part of the regular representation of $G$. Since $i_* \mathbb{Q}_{|Z}$ is supported at a point, this space is naturally identified with the stalk of $\cL_{nt}$ at $Z$, hence with $\mathbb{Q}[G]_{nt}$. Therefore, $\dim U = |G| - 1$.

To see algebraicity, note that $\widetilde{\pi} : W \to \mathrm{Sym}^h(C')$ is finite \'etale, so the inverse image of $Q$ is a disjoint union of $|G|$ irreducible components $Q_t$, permuted transitively by $G$. Their cohomology
classes $[Q_t]$ are linearly independent. The $G$-invariant sum $\sum_{t \in G} [Q_t]$ corresponds to the trivial representation, and its orthogonal complement inside
their span identifies with $\mathbb{Q}[G]_{nt}$. By Proposition \ref{prop_Theorem1(1)}, this complement is precisely $U$. Thus $U$ is generated by algebraic classes and has dimension $|G|-1$.
\end{proof}

\begin{example}
    Suppose $g(C')=3$, and $G$ is an abelian group. Then $h=4$. The Abel--Jacobi map $\Sym^4(C')\to J(C')$ is generically a $\mathbb P^1$-bundle, and has special fiber $\mathbb P^2$ over the canonical class.
    In this case, $\rH^4(\Sym^4(C'),\bbQ)\cong \rH^4(J(C'),\bbQ)\oplus \rH^2(J(C'),\bbQ)\oplus \bbQ$, where the last summand is generated by the class of the special fiber. Accordingly, $\rH^4(W,\bbQ)\cong \rH^4(A,\bbQ)\oplus \rH^2(A,\bbQ)\oplus U$ where $U\cong\bbQ^{|G|-1}$ is generated by the preimage of $\mathbb P^2=|K_{C'}|$ modulo the diagonal vector (i.e. $\sum_{t\in G} [Q_t]$ in the notation above).
\end{example}

We now prove the last part of Theorem \ref{thm:mainthm1}.
\begin{prop}
 The isomorphism \eqref{eqn_prop3.1} is $G$-equivariant, and the $G$-action on the summand $$\rH^*(\Sym^h(C'),\bbQ)$$ is trivial. The $G$-action on $U$ is the rational representation whose complexification is the direct sum of all non-trivial characters of $G$ (each appearing with multiplicity one).
\end{prop}
\begin{proof}
The $G$-equivariance follows by keeping track of the morphisms in the proof of Proposition \ref{prop_Theorem1(1)}, which are all $G$-equivariant. The second part follows from the fact that the summand $\rH^*(\Sym^h(C'),\bbQ)$ comes from the pullback and corresponds to the trivial representation factor in the decomposition of $\tilde{\pi}_*\bbQ$. The factor $U$ corresponds to the cohomology of local system $\tilde{\mathcal{L}}_{nt}$, which is identified with $\bbQ[G]_{nt}$, as explained in the proof of Proposition \ref{prop_Theorem1(2)}. The result then follows from the fact that $G$ is abelian and therefore $U\otimes \C\cong \C[G]_{nt}$ decomposes as sums of one-dimensional representations $\oplus_{\chi\neq0}\C[G]_{\chi}$.
\end{proof}

\begin{remark}
In fact, one can explicitly write down the basis of the $\chi$-isotypic component $\rH^h(W,\bbQ)_{\chi}=(U\otimes \C)_{\chi}$ for a given nonzero character $\chi$. We may consider the cohomology class 
\begin{equation}\label{eqn_chi_*Q_0}
    [\chi_*Q_0]:= \sum_{t\in G} \chi(-t)[Q_t].
\end{equation}
Then this element lies in $\rH^{h}(W,\bbQ(\mu_m))$ where $m = |G|$ since the field of definition of $\chi$ is contained in $\bbQ(\mu_m)$. On the other hand, direct computation shows that $s_*[\chi_*Q_0] = \chi(s)\sum_{t\in G}\chi(-st)[Q_{st}]=\chi(s)[\chi_*Q_0]$ for all $s\in G$. Hence, \eqref{eqn_chi_*Q_0} lies in the $\chi$-isotypic component. 
\end{remark}

\begin{example}\normalfont
    Let $G=\Z/3\Z$ be the cyclic group of order 3, then there are two non-trivial characters $\chi_k$, with $k=1,2$ so that $\chi_k(1)=\mu_3^k$, where $\mu_3=\exp(2\pi i/3)$ is the third root of unity. Denote by $Q_0,Q_1,Q_2$ the $G$-orbit of the algebraic cycle $Q_0=\mathbb P^{g-1}$ as above. Then $U\otimes \C$ has $\chi$-isotypic basis  $Q_{\chi_1}=Q_0+\mu_3^{-1}Q_1+\mu_3^{-2}Q_2,$ and $Q_{\chi_2}=Q_0+\mu_3^{-2}Q_1+\mu_3^{-1}Q_2.$ Since $Q_{\chi_1}+Q_{\chi_2}=2Q_0-Q_1-Q_2$ and $\frac{i}{\sqrt{3}}(Q_{\chi_1}-Q_{\chi_2})=-Q_1+Q_2$, we see that their span has rational structure.
\end{example}

\begin{remark}
  Note that in the proofs above we could consider any rational representation of $G$. In particular, one can consider instead $J(C)_{V_i}$ (for each rational representation $V_i$), and analogous subspaces $U_i$. 
      This allows one to obtain a version of \ref{thm:mainthm2} where $J(C)_{nt}$ is replaced by $J(C)_{V_i}$.  
\end{remark}

\section{Algebraic Cycles on Prym Variety}\label{sec_proof-thm2}
In this section, we will show that the Hodge substructure $U$ from Theorem \ref{thm:mainthm1} can be identified with the Weil Hodge structure $U_{Weil}$ in the cohomology of the Prym variety $B$. As a consequence, we find that the algebraic cycles that generate $U$ produce algebraic cycles that generate $U_{Weil}$.

\subsection{A factorization of the Abel--Jacobi map}

 First, we note that diagram \ref{diagram_W-A} extends to the Abel--Jacobi map 
 \begin{equation}\label{eqn_AJ_C}
     \Phi:\Sym^h(C)\to J(C)
 \end{equation}
 for the curve $C$.
\begin{lemma}
 The map $f:\Sym^h(C)\to \Sym^h(C')$  factors through $W$. In particular, we have the following commutative diagram. 
\end{lemma}

\begin{figure}[ht!]
    \centering
\begin{equation}\label{diagram_Wcoh}
\begin{tikzcd}
\Sym^h(C)\arrow[d,"\phi"]\arrow[r,"\Phi"]\arrow[bend right=25, swap, "f"]{dd}& J(C)\arrow[d]\arrow[bend left=25, "\textup{Nm}"]{dd}\\
W\arrow[d,"\tilde{\pi}"] \arrow[r] &  A \arrow[d,"\pi"'] \\
\Sym^h(C') \arrow[r,"AJ_h"] & J(C')
\end{tikzcd}
\end{equation}
\end{figure}

\begin{proof}
The lower diagram results from geometric class field theory (Theorem \ref{thm:gcft}) as before. The outer diagram is induced by Abel--Jacobi maps and the $G$-cover $C\to C'$: the map $f$ is induced by functoriality of the symmetric product, and $\textup{Nm}:J(C)\to J(C')$ is the norm map. The Stein factorization of $\textup{Nm}$ induces a morphism of abelian varieties $J(C)\to A$ with connected fibers.

Thus all maps in \eqref{diagram_Wcoh} are defined, and the diagram commutes. It remains to construct $\phi$. Since $W\cong\Sym^h(C')\times_{J(C')}A$, the fiber product $\Sym^h(C')\times_{J(C')} J(C)$ is isomorphic to $W\times_A J(C)$. Hence, by the universal property of the fiber product, there is a map $\Sym^h(C)\to W\times_A J(C)$, and composing with the projection to $W$ yields the desired morphism $\phi$. 
\end{proof}

\subsection{Group actions}
In this section, we describe various groups that act on the diagram \ref{diagram_Wcoh}. These result from the action of $G$ on $C$.

Since $C\to C'$ is an etale cover with deck transformation group $G$, one has $G\subseteq\Aut(C)$. Hence, the automorphism group $\Aut(C^h)$ of the product $C^h$ contains the product $G^h=G\times \cdots\times G$. Note that it also contains the permutation group $\mathfrak{S}_h$. We define the subgroup $N$ of $G^h$ as follows: 
$$N:=\{(t_1,\ldots,t_h)\in G^h\subseteq \Aut(C^h)|\sum_kt_k=0\}.$$
Then $N\trianglelefteq G^h$ is a normal subgroup. Let $H:=N\mathfrak{S}_h$, and $\tilde{H}:=G^h\mathfrak{S}_h$. One checks that there is a diagram \eqref{diagram_W0} of inclusions of normal subgroups (except for $\mathfrak{S}_h\subseteq G$) together with the corresponding quotients of $C^h$.
\begin{figure}[ht]
    \centering
\begin{equation}\label{diagram_W0}
\begin{tikzcd}
\{0\}\arrow[d,symbol=\trianglelefteq] \arrow[r,symbol=\trianglelefteq] &  \mathfrak{S}_h \arrow[d,symbol=\subseteq] &&C^h\arrow[d,"/N"] \arrow[r,"/\mathfrak{S}_h"] \arrow[dr,"/H"]&  \Sym^hC \arrow[d,"\phi"] \\
N\arrow[d,symbol=\trianglelefteq] \arrow[r,symbol=\trianglelefteq] & H\arrow[d,symbol=\trianglelefteq]&&C^h/N\arrow[d,"/G"] \arrow[r,"/\mathfrak{S}_h"] & W\arrow[d,"/G"]\\
 G^h\arrow[r,symbol=\trianglelefteq] & \tilde{H},&& C'^h\arrow[r,"/\mathfrak{S}_h"] & \Sym^h(C').
\end{tikzcd}
\end{equation}
\end{figure}

We note that the diagram \ref{diagram_W0} extends the diagram \ref{diagram_Wcoh} to the left. Note that $\phi:\Sym^h(C)\to W$ is a finite morphism, but is not obtained as a quotient by a group action.



\subsection{Macdonald's formula}
By Macdonald's result \cite{Macdonald}, there is a decomposition
\begin{equation}\label{eqn_macdonold}
    \rH^h(\Sym^hC,\bbQ)\cong \bigoplus_{0\le i\le h/2}\bigwedge^{h-2i}\rH^1(C,\bbQ)(-i),
\end{equation}
where the wedge power is taken over $\bbQ$, and the isomorphism preserves the Hodge structures. 

Since $\phi$ is finite and the Hodge substructure $U$ is algebraic (cf. Theorem \ref{thm:mainthm1}), the pullback $\phi^*U$ is algebraic as well. We now determine in which summand of \eqref{eqn_macdonold} the pullback $\phi^*U$ lies. The result is obtained by keeping track of the various group actions above.
\begin{lemma} \label{lemma_phi*U}
The pullback
$$\phi^*:\rH^h(W,\bbQ)\to \rH^h(\Sym^h(C),\bbQ)$$
maps the Hodge substructure $U$ (cf. Theorem \ref{thm:mainthm1}) into the factor $\bigwedge^h\rH^1(C,\bbQ)$.
\end{lemma}
\begin{proof}
Consider the following diagram coming from diagram \ref{diagram_W0}.
\begin{figure}[ht]
    \centering
\begin{equation}
\begin{tikzcd}
C^h \arrow[r,"/\mathfrak{S}_h"] \arrow[dr,"\tilde{\phi}"']&  \Sym^hC \arrow[d,"\phi"]\\
 & W
\end{tikzcd}
\end{equation}
\end{figure}

Then we note that $\tilde{\phi}^*U$ factors through $\phi^*U$. Since both $W$ and $\Sym^h(C)$ arise from quotients of $C^h$ by the groups $H$ and $\mathfrak{S}_h$, respectively, their cohomology groups are identified with subspaces of $\rH^h(C^h,\bbQ)$ fixed by the corresponding groups. Hence, we have the following commutative diagram of inclusions:

\begin{figure}[ht]
    \centering
\begin{equation}
\begin{tikzcd}
\rH^h(C^h,\bbQ)  &  \rH^h(C^h,\bbQ)^{\mathfrak{S}_h} \arrow[l,hookrightarrow] \\
& \rH^h(C^h,\bbQ)^{H}\arrow[ul,hookrightarrow,"\tilde{\phi}^*"] \arrow[u,"\phi^*",hookrightarrow].
\end{tikzcd}
\end{equation}
\end{figure}

Thus, to determine the image of $\phi^*U$, it suffices to understand the action of the group on each summand of \eqref{eqn_macdonold}.

Consider the subgroup 
$$G_0:=\{(g,-g,0,\ldots,0)\in G^h|g\in G\}$$
of $G^h$, which is isomorphic to $G$. Note that the composition 
$$G_0\to G^h\mathfrak{S}_h\to \tilde{H}/H\cong G$$
is an isomorphism, and the induced $G_0$ action on $U$ is precisely the standard $G$-action.

We will investigate the $G_0$-action on each K\"unneth factor of $\rH^h(C^h,\bbQ)$, which is given by
$$\big (\rH^1(C)\otimes\cdots\otimes \rH^1(C)\big )\oplus \big (\rH^0(C)\otimes \rH^2(C)\otimes \cdots\big )\oplus\cdots.$$

In the first term, the $G_0$ action is faithful, so its induced action on the symmetrization $\wedge^h\rH^1(C,\bbQ)$ is faithful as well. All other factors have at least one copy of $\rH^0(C)$ and $\rH^2(C)$. In these factors, the group $G_0$ or its composition with certain permutations acts trivially. Their symmetrization corresponds to the factors $\wedge^{h-2i}\rH^1(C,\bbQ)$, with $i\ge 0$. Hence, the $G_0$-action on these factors is not faithful. By Schur's lemma, the projection of $\tilde\phi^*U$ to these summands must therefore vanish. Hence, $\tilde\phi^*U$ is contained in $\rH^1(C)^{\otimes h}$. Since $\phi^*U$ lies in the $\mathfrak{S}_h$-invariant part, and the $\mathfrak{S}_h$-invariants in $\rH^1(C,\bbQ)^{\otimes h}$ identify with $\bigwedge^h \rH^1(C,\bbQ)$, it follows that $\phi^*U$ maps into the factor $\bigwedge^h \rH^1(C,\bbQ)$.
\end{proof}

The rest of the section is devoted to proving the second main theorem.

\subsection{Proof of Theorem \ref{thm:mainthm2}}

\begin{thm}
     The Weil Hodge classes $U_{Weil}=\Lambda^h_{R}\rH^1(B,\bbQ) \subset \rH^{h}(B,\bbQ)$ are represented by algebraic cycles. 
\end{thm}
\begin{proof}
First, we have the following commutative diagram of cohomology groups with rational coefficients: 

\begin{figure}[ht!]
    \centering
\begin{equation}
\begin{tikzcd}
\phi^*U\arrow[r,symbol=\subseteq] \arrow[bend left=20,swap]{drrr}& \rH^h(Sym^hC,\bbQ) \arrow[dr,"\Phi_*"] \\ 
U^h\arrow[r,symbol=\subseteq]\arrow[u,"\cong"]\arrow[bend right=13,swap]{rrr} & \rH^h(J(C),\mathbb Q)\arrow[r,"L^{g(C)-h}","\cong"']\arrow[u,"\Phi^*"]&\rH^{2g(C)-h}(J(C),\mathbb Q)\arrow[r,symbol=\supseteq]& U^{2g(C)-h}.
\end{tikzcd}
\end{equation}
\end{figure}

The commutativity of the inner diagram follows from the fact that the Abel--Jacobi image of the $h$-th symmetric product is the Brill--Noether locus $W_h$ (note by Riemann--Roch $g(C)> |G|(g(C')-1)\ge h$ as long as $G$ is non-trivial, hence $\Phi$ is birational onto its image $W_h$), whose cohomology class is given by $\frac{1}{(g(C)-h)!}[\Theta]^{g(C)-h}$, where $[\Theta]$ is the theta divisor inducing the principal polarization on $J(C)$. By the projection formula, the composition $\Phi_*\Phi^*$ is, up to a scalar, given by powers of the Lefschetz operator. Hence, the composite $\Phi_*\Phi^*$ induces an isomorphism.

The outer diagram induces isomorphisms on Hodge substructures, where
\begin{itemize}
    \item $U^h=\bigwedge^h_R \rH^1(C,\bbQ)_{nt}$,
    \item $U^{2g(C)-h}$ is the image of $U^h$ under Lefschetz operator, and  
    \item the pullback $\phi^*U$ (algebraic) Hodge substructure from Lemma \ref{lemma_phi*U} coincides with $\Phi^*U^h$.
\end{itemize}

The last item follows from Lemma \ref{lemma_phi*U} that $\phi^*U$ is contained in the factor $\bigwedge^{h}\rH^1(C,\bbQ)$ of cohomology of the symmetric product. Recall that there is a natural isomorphism 
$$\bigwedge^{h}\rH^1(C,\bbQ)\cong \rH^h(J(C),\bbQ).$$
Hence, $\phi^*U$ must agree with $\Phi^*U^h$.

Next, we prove the following Lemma.
\begin{lemma}
  The pushforward along the Abel--Jacobi map $\Phi_*$ induces an isomorphism of Hodge structures
   $$\phi^*U\cong U^{2g(C)-h},$$
  and the target $U^{2g(C)-h}$ is generated by algebraic cycles.
\end{lemma}
\begin{proof}
  From the proof of Proposition \ref{prop_Theorem1(2)}, $U$ is algebraic and generated by classes of disjoint cycles $Q_t\cong \mathbb P^{g-1}$, with $t\in G$. Hence, the Hodge substructure $\phi^*U$ is algebraic and represented by the preimages $\phi^{-1}(Q_t)$. Therefore, the Abel--Jacobi map images of these cycles represent $U^{2g(C)-h}$.
\end{proof}

We recall that from Section \ref{sec_prelim}, the Prym variety $B$ is the subvariety of $J(C)$ associated with $V_{nt}$, the sum of all non-trivial $G$-representations. Note that $J(C')$ is identified with the abelian subvariety $J(C)$ associated with the trivial representation. It follows that we have an isogeny $\tau$ as in the following diagram, where $p_i$ is the projection to the $i$-th factor.

\begin{figure}[ht]
    \centering
\begin{equation}
\begin{tikzcd}
J(C) \arrow[r,symbol=\sim,"\tau"] &  J(C')\times B \arrow[r,"p_1"]\arrow[d,"p_2"]&J(C')\\
 & B.
\end{tikzcd}
\end{equation}
\end{figure}
 
By the commutativity of the diagram \ref{diagram_Wcoh}, the image of each of the $\phi^{-1}(Q_t)$ representing $\phi^*(U)$ via composition $p_1\circ\tau:J(C)\to J(C')$ is a point. Hence via the isogeny $\tau$,  $U^{2g(C)-h}$ is identified with the Hodge substructure $U_B^{2\dim(B)-h}$ in the K\"unneth component
$$\rH^{2g(C')}(J(C'))\otimes \rH^{2g(C)-h-2g(C')}(B).$$

Note that $g(C)-g(C')=\dim(B)$. Therefore the above is identified with $\rH^{2\dim(B)-h}(B)$. Hence the composition
$p_2\circ\tau:J(C)\to B$
sends each of the algebraic cycles $\phi^{-1}(Q_t)$ to algebraic cycles in $B$ and their Poincar\'e duals generate the Hodge substructure $U_B^{2\dim(B)-h}$.

Finally, since the Lefschetz standard conjecture is known for abelian varieties (cf. \cite{Kleiman}), the Lefschetz inverse operator $L^{-1}$ is algebraic, and so are its higher powers. In other words, the morphism of Hodge structures $L^{-k}:\rH^{n+k}(A,\bbQ)\to \rH^{n-k}(A,\bbQ)$ is algebraic for each $k$, where $A$ is an abelian variety of dimension $n$. In particular,
$$U_B^h:=L^{-(\dim(B)-h)}(U_B^{2\dim(B)-h})$$
is algebraic. Note that $U_B^h$ is precisely the Hodge substructure identified with the Weil Hodge structure \eqref{eqn_UWeil}. Therefore, the claim follows.
\end{proof}

\section{Relation to Schoen's Setting}\label{sec_cyclic}

We now explain how the results of the previous sections specialize to Schoen's setting \cite{Schoen88}. Let $C\to C'$ be an unramified cyclic cover of smooth projective curves, with Galois group $G=\Z/m\Z$.

Recall that the Prym variety $B$ (cf. \eqref{eqn_Prym}) is the abelian subvariety of $J(C)$ defined as the connected component of the identity in the kernel of the Norm map. In our notation, it is $J(C)_{nt}$, the abelian subvariety of $J(C)$ associated with the $G$-sub-representation $\rH^1(C,\bbQ)_{nt}$. Equivalently, the Prym variety can be written as 
\begin{equation}\label{eqn_PrymNaive}
B={\rm Im}(1-\sigma):J(C)\to J(C),
\end{equation}
where $\sigma:J(C)\to J(C)$ is induced by a generator of $G$.

We recall some standard properties of the Prym variety $B$ (see \cite{Faber}).

\begin{lemma}
    The Prym variety $B$ has the following properties:
\begin{enumerate}
    \item $B$ has dimension $(m-1)(g(C')-1)$.
    \item The generator $\sigma$ above induces an automorphism of $B$ of order $m$. The tangent space at the identity $T_0B$ has a natural induced action of $\Z/m\Z$ for which each eigenvalue $\exp(2\pi i k/m)$ with $1\le k\le m-1$ occurs with multiplicity $g(C')-1$.
\end{enumerate}
\end{lemma}
 The first property follows from Lemma \ref{lem:chimultiplicity}, which shows that all nontrivial characters occur with the same multiplicity. The second property comes from the fact (proven in Lemma \ref{lemma_chi-chi-bar}) that each non-trivial character $\chi$ appears in balanced Hodge type, i.e., half of its contribution lies in $\rH^{1,0}(B,\C)$ and the other half in $\rH^{0,1}(B,\C)$. Alternatively, it is proved in \cite[Lemma 1.6a]{Schoen88} using the holomorphic Lefschetz fixed point theorem.

When $m$ is prime, $\rH^1(B,\bbQ)$ is a free module over $\bbQ(\mu_m)$, and the Prym variety considered in \cite{Schoen88} coincides with \eqref{eqn_Prym}. When $m$ is not prime, however, $B$ decomposes into several isotypic factors. In this case, Schoen's Prym variety \cite[p.24]{Schoen88} appears as an irreducible factor of $B$. We will explain the construction below. 

\begin{definition}
    Let $G=\Z/m\Z$ be a cyclic group. Then we call a character $\chi$ \textit{primitive} if $\chi:G\to \C^*$ is injective.
\end{definition}

 In our notation, Schoen's Prym variety is the abelian subvariety $J(C)_{V_{\textup{prim}}}$ associated with the irreducible rational $G$-representation $V_{\textup{prim}}$, whose complexification consists of the primitive characters. With this notation, the abelian subvariety $B_{\textup{prim}}:=J(C)_{V_{\textup{prim}}}$ of $J(C)$ satisfies the following properties:

\begin{itemize}
    \item The cohomology $\rH^1(B_{\textup{prim}},\bbQ)$ is naturally a vector space over $\bbQ(\mu_m)$, and the tangent space at the origin $T_0B_{\textup{prim}}\cong F^1\rH^1(B_{\textup{prim}},\C)^{\vee}$ decomposes as a direct sum of eigenspaces corresponding to the primitive characters, with each character appearing with multiplicity $h/2$.
    \item $\End_{\bbQ}(B_{\textup{prim}})\supseteq \bbQ(\mu_m)$.
\end{itemize}

In this setting, Schoen showed that each primitive character occurs with multiplicity $h=2g(C')-2$. In fact, Lemma \ref{lem:chimultiplicity} shows that this holds for every non-trivial character, and that the resulting Hodge substructure $$U_{\textup{prim}}:=\bigwedge^h_{\bbQ(\mu_m)}\rH^1(B_{\textup{prim}},\bbQ)\subseteq \rH^{h}(B_{\textup{prim}},\bbQ)$$
consists entirely of Hodge classes. Note it is a summand of \eqref{eqn_intro_Weil-classes}. Finally, Schoen proves the following theorem:

\begin{thm} \cite[Corollary 3.1]{Schoen88} $U_{\textup{prim}}$ is generated by algebraic cycles.
\end{thm}

This becomes a consequence of Theorem \ref{thm:mainthm2}.

\bibliographystyle{plain}
\bibliography{bibfile}

\end{document}